\documentclass[11pt,twoside]{article}

\usepackage{mathrsfs,amsfonts,amsmath,amssymb}
\usepackage{latexsym}
\newtheorem{satz}{Theorem}
\newtheorem{proposition}[satz]{Proposition}
\newtheorem{theorem}[satz]{Theorem}
\newtheorem{lemma}[satz]{Lemma}

\newtheorem{corollary}[satz]{Corollary}
\newtheorem{remark}[satz]{Remark}

\def\F{\mathbb {F}}
\def\E{\mathsf{E}}

\def\a{\alpha}

\def\C{\mathbb{C}}
\def\P{{\cal P}}

\def\({\big (}
\def\){\big )}

\def\G{\Gamma}

\def\le{\leqslant}
\def\ge{\geqslant}
\def\_phi{\varphi}
\def\eps{\varepsilon}

\def\Gr{{\mathbf G}}

\def\t{\tilde}

\def\D{\Delta}

\def\T{\mathsf{T}}

\def\C{\mathbb{C}}

\author{Olmezov K.I., Semchankau A.S., Shkredov I.D.}
\title{On popular sums and differences of sets with small products 
\footnote{This work is supported by the Russian Science Foundation under grant 19--11--00001.}
}
\date{}
\begin{document}
	\maketitle


\begin{center}
	Annotation.
\end{center}

{\it \small
	Given a subset of real numbers  $A$ with small product $AA$ we obtain a new upper bound for the additive energy of $A$.
	The proof uses a natural 
	observation that level sets of convolutions of the characteristic function of $A$ have small product with $A$.  
}
\\

\section{Introduction}
\label{sec:introduction}

Let $p$ be a prime number, $\F_p$ be the finite field, and let $\G \subseteq \F_p \setminus \{0\}$ be a multiplicative subgroup. 
The question about additive properties of such subgroups is a classical one, see, e.g., \cite{Bourgain_more}, \cite{Bourgain_DH}, \cite{K_Tula}, \cite{KS1} and many other papers. 
The discussed  question is naturally connected with the sum--product phenomenon, see, e.g, \cite{Elekes2}, \cite{TV} and recent papers \cite{RSS}, \cite{Shakan}. 
In many papers, see \cite{Bourgain_more}, \cite{BCh}, \cite{s_as} and others,  authors extensively exploit the fact that the sumsets and the difference sets of $\G$ 
are also $\G$--invariant sets, that is can be expressed as a disjoint union of some cosets over $\G$. Moreover, some more difficult functions as convolutions of $\G$, its level--sets and 
many others 
enjoy  this property as well.
The aim of this paper is to discuss what can be done in this direction in the real setting. 
One of our results says 
that if $A\subseteq \mathbb{R}$  be a set with small product, then any level set $P$ of its convolutions is almost invariant under multiplication by $A$ (the exact formulation can be found in Section \ref{sec:small_AA}).  
Notice that the first results in this direction were obtained in \cite{MRSS}, \cite{sh_AA_small}, \cite{s_as}. 
We apply the described observation to find a 
new bound 
for 
the additive energy of subsets in  $\mathbb{R}$  with small product set.

Let us recall quickly what was done before concerning the additive energy of such sets.
In \cite[Theorem 3]{MRSS}, developing a series of previous results (see, e.g., \cite{SS}, \cite{s_mixed}),  it was proved 

\begin{theorem}
	Let $A\subset \mathbb{R}$ be a finite set such that $|AA| \le M|A|$. 
	Then 
	\begin{equation}\label{f:main_intr_49/20}	
	\E^{+} (A) := |\{ (a,b,c,d) \in A^4 ~:~ a+b=c+d \}| \lesssim  M^{\frac{8}{5}} |A|^{\frac{49}{20}} \,.
	\end{equation}
	\label{t:main_intr-}
\end{theorem}

Using a combinatorial idea (see Section \ref{sec:small_AA}) as well as the eigenvalues method, we improve the last result to (consider the simplest case $M=1$)

\begin{theorem}
	Let $A\subset \mathbb{R}$ be a finite set such that $|AA| \ll |A|$. 
	Then 
\begin{equation}\label{f:main_intr}	
	\E^{+} (A) \lesssim 
	|A|^{\frac{22}{9}} \,.
\end{equation}
\label{t:main_intr}
\end{theorem}

Thus for any set $A$ with $|AA|\ll |A|$ we obtain the exponent $\frac{22}{9}$ which is better then in Theorem \ref{t:main_intr-}. 

Also, in \cite{sh_AA_small} and in \cite[Theorems 2, 10]{MRSS} the following result was proved (it is parallel to results from  \cite{SV} for multiplicative subgroups in $\F^*_p$).

\begin{theorem}
	Let $A\subset \mathbb{R}$ be a finite set such that $|AA| \ll |A|$. 
Then 
\begin{equation}\label{f:main_intr'}	
	|A-A| \gtrsim |A|^{\frac{5}{3}} \quad \quad \mbox{and} \quad \quad |A+A| \gtrsim |A|^{\frac{8}{5}} \,.
\end{equation}
\label{t:main_intr'}	
\end{theorem}

Here again we obtain  Theorem \ref{t:main_intr'} directly by the eigenvalues  method and applying our new combinatorial idea.

General discussion, partial results and open questions are contained in the  last Section \ref{sec:general_problem}.

We thank T. Schoen for  useful discussions.


\section{Definitions}
\label{sec:definitions}

Let $\Gr$ be an abelian group.
In this paper we use the same letter to denote a set $S\subseteq \Gr$
and its characteristic function $S:\Gr \rightarrow \{0,1\}.$
By $|S|$ we denote cardinality of $S$.
Given two sets $A,B\subset \Gr$, define  the \textit{product set} (the {\it sumset} in the abelian case) of $A$ and $B$ as 
$$AB:=\{ab ~:~ a\in{A},\,b\in{B}\}\,.$$
In a similar way we define the higher product sets, e.g., $A^3$ is $AAA$. 
If $\Gr$ is an abelian group, then the Pl\"unnecke--Ruzsa inequality (see, e.g., \cite{TV}) takes place
$$
	|nA-mA| \le \left( \frac{|A+A|}{|A|} \right)^{n+m} \cdot |A| \,.
$$
Let $f,g : \Gr \to \C$ be two functions.
Put
\begin{equation}\label{f:convolutions}
(f*g) (x) := \sum_{y\in \Gr} f(y) g(x-y) \quad \mbox{ and } \quad
(f\circ g) (x) := \sum_{y\in \Gr} f(y) g(y+x) \,.
\end{equation}
Denote by 
$\E^{+}(A,B)$ the {\it additive energy} of two sets $A,B \subseteq \Gr$
(see e.g. \cite{TV}), that is
$$
\E^{+} (A,B) = |\{ (a_1,a_2,b_1,b_2) \in A^2 \times B^2 ~:~ a_1+b_1 = a_2+b_2  \}| \,.
$$
If $A=B$ we simply write $\E^{+} (A)$ instead of $\E^{+} (A,A).$
Clearly,
\begin{equation*}\label{f:energy_convolution}
\E^{+} (A,B) = \sum_x (A*B) (x)^2 = \sum_x (A \circ B) (x)^2 = \sum_x (A \circ A) (x) (B \circ B) (x)
\,.
\end{equation*}
More generally (see, e.g., \cite{SS}, \cite{s_E_k}), for $k\ge 2$ put
$$
\E^{+}_k (A) = |\{ (a_1,\dots, a_k, a'_1, \dots, a'_k) \in A^{2k} ~:~ a_1-a'_1 = a_2 - a'_2 = \dots = a_k - a'_k  \}|
\,.
$$
Thus $\E^{+} (A) = \E^{+}_2 (A)$.
It is convenient to put $\E^{+}_1 (A) = |A|^2$.  
Having $A,P \subseteq \Gr$ let $\sigma_P (A) := \sum_{x\in P} (A\circ A) (x)$.
In the same way define the {\it multiplicative energy} of two sets $A,B \subseteq \Gr$
$$
\E^{\times} (A,B) = |\{ (a_1,a_2,b_1,b_2) \in A^2 \times B^2 ~:~ a_1 b_1 = a_2 b_2 \}|
$$
and, similarly, $\E^\times_k (A)$.
Certainly, the multiplicative energy $\E^{\times} (A,B)$ can be expressed in terms of multiplicative convolution 
as in
(\ref{f:convolutions}).
If it does not matter which energy we need, then let us  write just $\E(A)$, $\E_k(A)$ and so on. 
Also, sometimes we  use representation function notations like $r_{AB} (x)$, $r_{A+B} (x)$ or $r_{AB^{-1}} (x)$, which counts the number of ways $x \in \Gr$ can be expressed as a product $ab$ or as a sum $a+b$ or  $ab^{-1}$ with $a\in A$, $b\in B$, respectively. 
For example, $|A| = r_{AA^{-1}}(1)$ and  $\E (A,B) = r_{AA^{-1}BB^{-1}}(1) =\sum_x r^2_{A^{-1}B} (x)$. 


All logarithms are to base $2.$ The symbols  $\ll$ and $\gg$ are the usual Vinogradov's symbols, thus $a\ll b$ means $a=O(b)$ 
and $a\gg b$ is $b=O(a)$. 
If $K$ is a parameter, then $\ll_K$, $\gg_K$ 
indicates  
a polynomial dependence of constants in $\ll$ and $\gg$ on $K$. 
Having a fixed  set
$A$,
we write
$a \lesssim b$ or $b \gtrsim a$ if $a = O(b \cdot \log^c |A|)$, with an absolute constant  $c>0$.
For any given prime $p$ denote by $\F_p$ the finite prime field and let $\F$ be an arbitrary field no matter finite or not.  

\section{Preliminaries}
\label{sec:preliminaries}

The first lemma is a well--known consequence of the Szemer\'edi--Trotter Theorem (see below) and is contained  in, e.g., \cite[Corollary 28]{SS}. 
It says that a certain sort of energy of a set $A$ with small product set $AA$ can be estimated almost optimally.

\begin{lemma}
	Let $A\subset \mathbb{R}$ be a finite set such that $|AA| \le M|A|$.
	Then 
\[
	\E^{+}_3 (A) \ll M^2 |A|^3 \log |A| \,. 
\] 	
\label{l:E_3}
\end{lemma}

The next lemma is a partial case of the eigenvalues method (although inequalities \eqref{f:eigenvalues1.5}, \eqref{f:eigenvalues2} can be obtained directly using a purely combinatorial approach), see, e.g., \cite[Theorem 5.1, inequality (5.7)]{s_mixed}.

\begin{lemma}
	Let $\Gr$ be an abelian  group and $A\subset \Gr$ be a finite set.
	Then for an arbitrary set $P \subseteq A-A := D$ such that for any $x\in P$ one has $\D < (A\circ A) (x) \le 2\D$
	the following holds 
\begin{equation}\label{f:eigenvalues}
	\left( \frac{\sigma^2_P (A) \E(A)}{|A|^3} \right)^2 \lesssim \E_3 (A) \cdot \sum_{x,y} (A\circ A)^2 (x-y) P(x) P(y) \,. 
\end{equation}
	Similarly, for any $P\subseteq D$ one has 
\begin{equation}\label{f:eigenvalues1.5}
	\left( \frac{\sigma^2_P (A)}{|A|} \right)^2 \le \E_3 (A) \cdot \sum_{x,y} D (x-y) P(x) P(y) \,. 
\end{equation}
	In particular, 
\begin{equation}\label{f:eigenvalues2}
	|A|^6 \le \E_3 (A) \cdot \sum_{x,y} D(x-y) D(x) D(y) \,. 
\end{equation}
\label{l:eigenvalues}
\end{lemma}
	

The next lemma is a small generalization of   Exercise 1.1.8 from \cite{TV} and can be obtained using the probabilistic  method, say, combined with the Pl\"unnecke--Ruzsa inequality. 

\begin{lemma}
	Let $A,B \subseteq \Gr$ be two finite sets.
	Then there exists a set $X\subseteq A+B-B$,
	$$
	|X| \ll \frac{|A+B-B|}{|B|} \cdot \log |A+B|
	$$
	such that
	$A+B \subseteq X+B$.
	In particular, for $B=A$ one has $A+A \subseteq X+A$ and 
$$
	|X| \ll \frac{|A+A|^3}{|A|^3} \cdot \log |A| \,.
$$
	\label{l:Schoen_covering}
\end{lemma}

We need the famous Szemer\'edi--Trotter Theorem \cite{ST} about incidences of points and lines on the plane. 
Let us recall the definitions.
Let $\mathcal{L}$ be a finite set of lines on the Euclidean plane and $\P$ be a finite ensemble of points.  
Define the {\it number of incidences} $\mathcal{I} (\mathcal{P},\mathcal{L})$ between points and lines  as
$\mathcal{I}(\mathcal{P},\mathcal{L})=|\{(p,l)\in \mathcal{P}\times \mathcal{L} : p\in l\}|$.

\begin{theorem}\label{t:SzT}
	Let $\mathcal{P}$ be a finite set of points and let $\mathcal{L}$ be a finite set of lines.
	Then
	$$\mathcal{I}(\mathcal{P},\mathcal{L}) \ll |\mathcal{P}|^{2/3}|\mathcal{L}|^{2/3}+|\mathcal{P}|+|\mathcal{L}|\,.$$
\end{theorem}

\section{On sums and differences of sets with small product set}
\label{sec:small_AA}

First of all, let us consider the following basic question. 
Suppose that $A$ is a finite subset of a field $\F$ and $|AA| \lesssim |A|$. 
Is it true that $|(A+A) A| \lesssim |A+A|$ or, in a similar way, $|(A-A) A| \lesssim |A-A|$?
If $A$ is a multiplicative subgroup of $\F_p$, then the answer is, obviously, positive but what if $A$ belongs to an infinite field, say, $\mathbb{R}$ 
where there are no pure nontrivial  subgroups?  
Below we will give 
an affirmative 
answer considering popular subsets of $A+A$ and $A-A$.
It is interesting that the answer to the dual question, namely, is it true that $|A+A| \lesssim |A|$ implies  $|AA+A| \lesssim |AA|$ or  $|A/A + A| \lesssim |A/A|$ is clearly negative (consider a shifted interval, e.g.).

Let $A$ be a set and put $D=A-A$, $\Pi = AA$. 
Suppose that $\D>0$ is a positive number and $P\subseteq D$ is a set such that $\D \le  r_{A-A} (x)$ for all $x\in P$. 
Then for any $x\in PA$ one has $r_{\Pi-\Pi} (x) \ge \D$ because the formula $xa = (a_1-a_2) a = a_1 a - a_2 a \in \Pi -\Pi$. 
Thus
\[
	\D |PA| \le \sum_{x\in PA} r_{\Pi-\Pi} (x) \le \sum_x r_{\Pi-\Pi} (x) = |AA|^2 = M^2 |A|^2 \,.
\]
It follows that 
\begin{equation}\label{f:PA}
	|PA| \le \frac{|AA|^{2}}{\D} = \frac{M^2 |A|^{2}}{\D} \,.
\end{equation}
Certainly, the same holds if one replaces $AA$ to $A/A$ (and even more general products as $AB$ can be used).

Now let us obtain another bound (previous logic was used in \cite{MRSS}, \cite{s_as}).
Suppose that, in addition to $\D \le r_{A-A} (x)$, that $r_{A-A} (x) \le 2\D$ on $P$ and for a certain  integer $k\ge 1$  
one has 
$\sum_{x \in P} r^k_{A-A} (x) \gtrsim \E^{+}_k (A)$. 
The previous arguments give us 
a generalization of \eqref{f:PA} 
$$
|PA| \D^k \le \sum_{x \in PA} r^k_{\Pi-\Pi} (x) \le \E^{+}_k (AA)  \,.
$$
Further by Lemma \ref{l:Schoen_covering}, we have $AA\subseteq XA$ and $|X| \lesssim M^3$. 
Hence by the norm property of the higher energies $\E_k$, see, e.g., \cite[Section 4]{s_E_k} and the definition of the set $P$, we get 
\begin{equation}\label{tmp:06.11_3} 
|PA| \D^k \le 
\E^{+}_k (AA) \le \left( \sum_{x \in X} (\E^{+}_k (xA))^{1/2k} \right)^{2k} = |X|^{2k}  \E^{+}_k (A)
\lesssim_k M^{6k} \E^{+}_k (A) 
\lesssim
\end{equation}
\begin{equation}\label{tmp:06.11_3'} 
\lesssim
M^{6k} \sum_{x \in P} r^k_{A-A} (x) \le  2^k M^{6k} |P| \D^k 
\,.
\end{equation}
It means that 
\begin{equation}\label{f:popular_PA}
|PA| \lesssim_{2^k} M^{6k} |P|  
\,.
\end{equation}
Interestingly, that we cannot replace the addition to multiplication and vice versa in \eqref{tmp:06.11_3}, \eqref{tmp:06.11_3'}.
For $k=1$ one can easily see that a slightly stronger bound takes place (compare it with \eqref{f:PA})  for any set $P$  such that $\D |P| \gg |A|^2$, 
namely,  (here we do not use any norm property)
\begin{equation}\label{f:popular_PA_1}
|PA| \ll M^{2} |P|  
\,.
\end{equation}
Finally, notice that  the same calculations take place  if one replaces $A-A$ to $A+A$ and  the energies $\E^{+}_k$ to other energies which enjoy norm properties, e.g., $\T^{+}_k$, $\E^{+}_{k,l}$ and so on, see \cite{SS}, \cite{s_E_k}. 
Also, we can consider sets with $|A/A| \le M|A|$ as well but the dependence on $M$ in \eqref{f:popular_PA} will be slightly worse in this case.

\bigskip

Calculations above allows us to show that popular difference/sumsets $P$ defined via 
sets $A$ with small $AA$ are so--called Szemer\'edi--Trotter type sets, see \cite{sh_SzT}.

\begin{corollary}
	Let $A \subset \mathbb{R} \setminus \{0\}$ be a set, $|AA| \le M|A|$ and $P$ as above. 
	Then 
\begin{equation}\label{f:P_energy'} 	
	\E^{+}_3 (P) \lesssim_{2^k} \frac{M^{12k} |P|^4}{|A|} + |P|^3 \,,
\end{equation}
	and for any set $B\subset \mathbb{R}$ one has  
\begin{equation}\label{f:P_energy''} 	
	\E^{+} (P,B) \lesssim_{2^k} \frac{M^{6k} |P|^{3/2} |B|^{3/2}}{|A|^{1/2}} + |P| |B| \,.
\end{equation}
\label{c:P_energy} 	
\end{corollary}
\begin{proof}
	If $P=\{0\}$, then there is nothing to prove. 
	Let $\tau \ge 1$ be a real number.
	It is enough to 
	obtain 
	for a certain  $\tau \gg  1$ 
\begin{equation}\label{tmp:18.10_1}
	|\{ s ~:~  |\{ p - b = s ~:~  p\in P,\, b\in B \} | \ge \tau \}| \lesssim_{2^k} \frac{M^{12k} |B|^2 |P|^2}{|A|\tau^3} 
\end{equation}
	after that bounds \eqref{f:P_energy'}, \eqref{f:P_energy''} follow via simple summation. 
	Denote by $S_\tau$ the set from \eqref{tmp:18.10_1}  and our task is to find the required upper bound for cardinality of $S_\tau$. 
	We have 
\[
	\tau |S_\tau| |A| 
		\le 
			|\{ \pi a^{-1} - b = s ~:~  \pi \in PA,\, b\in B,\, s\in S_\tau,\, a\in A  \} | \,.
\]
	We interpret the last equation as points/lines incidences. 
	Here $\mathcal{P} = A^{-1} \times  S_\tau$ and lines from $\mathcal{L}$ are indexed by  coefficients  $(\alpha, \beta)$ from $PA \times B$. 
	Applying Theorem \ref{t:SzT},  we see that 
\[
	\tau |S_\tau| |A| \ll (|A| |S_\tau| |PA| |B|)^{2/3} + |S_\tau| |A| + |B| |PA| \,.
\] 
	If the first term dominates, then inequality \eqref{f:popular_PA} gives the required estimate \eqref{tmp:18.10_1} because $|PA| \lesssim_{2^k} M^{6k} |P|$. 
	The second term cannot be the largest one because then $\tau \ll 1$. 
	It remains to consider the case when the third term dominates.
	Then we should have 
\[
	\frac{|B| |PA|}{|A|} \gg \frac{|PA|^2 |B|^2}{|A| \tau^2}
\]
	because otherwise there is nothing to prove.	
	But clearly, $\tau \le \min \{|P|, |B|\}$ and hence choosing the absolute constant in $O(\cdot)$ to be large enough we arrive to a contradiction.  
	This completes the proof. 
$\hfill\Box$
\end{proof}

\section{Applications}
\label{sec:proof}

Now we are ready to prove Theorem \ref{t:main_intr} from the Introduction.

\begin{theorem}
	Let $A\subset \mathbb{R}$ be a finite set such that $|AA| \le M|A|$. 
	Then 
	\begin{equation}\label{f:main_energy}	
	\E^{+} (A) \lesssim M^{\frac{7}{3}} |A|^{\frac{22}{9}} \,.
	\end{equation}
	\label{t:main_energy}
\end{theorem}
\begin{proof} 
Let $P$ be a subset of $A-A$ such that for all $x\in P$ one has $\D < r_{A-A} (x) \le 2\D$ and $\sum_{x \in P} r^2_{A-A} (x) \gg \E^{+} (A) / \log |A|$. 
The existence of $P$ easily follows from the dyadic pigeon--holing  principle. 
From Lemma \ref{l:E_3} (up to a logarithm) or without logarithms, see \cite[Lemma 3.7 or the proofs of Theorems 5.1, 5.4]{s_mixed}, one has 
\begin{equation}\label{f:D_upper}
	\D \ll \frac{M^2 |A|^3}{\E^{+} (A)}
\end{equation}
Applying Lemma \ref{l:eigenvalues} and the definition of the set $P$, we have  
\[
	\left( \frac{\D^2 |P|^2 \E^{+}(A)}{|A|^3} \right)^2 
	\le \left( \frac{\sigma^2_P (A) \E^{+}(A)}{|A|^3} \right)^2 
	\lesssim  \E^{+}_3 (A) \cdot \sum_{x,y} r^2_{A-A} (x-y) P(x) P(y) 
	= 
\]
\[
	=
		\E^{+}_3 (A) \cdot \sum_z r^2_{A-A} (z) r_{P-P} (z) \,. 
\]
Using the H\"older inequality, we obtain
\[
	\left( \frac{\D^2 |P|^2 \E^{+}(A)}{|A|^3} \right)^6 
	\lesssim 
		(\E^{+}_3 (A))^5 \E_3^{+} (P) \,.
\]
In view of Lemma \ref{l:E_3} and Corollary \ref{c:P_energy} as well as our choice of $P$ and $k=2$, we have
\[
	\left( \frac{\D^2 |P|^2 \E^{+}(A)}{|A|^3} \right)^6 
\lesssim 
	|A|^{14} M^{34} |P|^4 \,.
\]
Hence, using \eqref{f:D_upper}, we get 
\[
	(\E^{+}(A))^{14} \lesssim M^{34} |A|^{32} \D^4 \ll  M^{34} |A|^{32}  \left(\frac{M^2 |A|^3}{\E^{+} (A)} \right)^4 
\]
and hence
\[
	\E^{+} (A) \lesssim M^{7/3} |A|^{22/9}
\]
as required.
$\hfill\Box$
\end{proof}

\bigskip 

In the same vein  we obtain Theorem \ref{t:main_intr'}. 
Let $D=A-A$ and $S=A+A$.
Choose $P\subseteq D$ such that $\sigma_P (A) \gtrsim |A|^2$ and  for a certain $\D$ one has  $\D < r_{A-A} (x) \le 2\D$ on $P$. 
Using inequality \eqref{f:eigenvalues1.5} of Lemma \ref{l:eigenvalues}, we get 
\[
	|A|^6 \lesssim \E^{+}_3 (A) \sum_{z\in D} r_{P-P} (z) 
\]
and hence by the H\"older inequality 
\[
	|A|^{18} \lesssim (\E^{+}_3 (A))^3 \E^{+}_3 (P) |D|^2 \,.
\]
Applying Lemma \ref{l:E_3}, 
combined with bound \eqref{f:popular_PA_1}, we derive
\begin{equation}\label{tmp:06.11.2019_1}
	|A|^{10} \lesssim M^{10} |D|^6
\end{equation}
as required. 
It is interesting that our bound \eqref{tmp:06.11.2019_1} coincides with the classical sum--product estimate of Elekes \cite{Elekes2} up to logarithms.

\bigskip

Similarly, by the proof of \cite[Theorem 11, inequality (4.9)]{sh_SzT}, we have 
\[
	|A|^{10} \lesssim |S|^2  \E^{+}_3 (A) \sum_{z} r^2_{A-A} (z) r_{S'-S'} (z) \,,
\]
where $S' \subseteq  \{ x ~:~ r_{A+A}(x) \ge |A|^2/(2|S|) \}$ and $\D < r_{A+A} (x) \le 2\D$ on $S'$. 
Using the H\"older inequality, we get 
\[
	|A|^{30} \lesssim |S|^6  (\E^{+}_3 (A))^5 \E^{+}_3 (S') \,.
\]
Applying Lemma \ref{l:E_3} and bound \eqref{f:popular_PA_1},
we derive
\[
	|A|^{16} \lesssim M^{14} |S|^{10} \,.
\]
This completes the proof of Theorem \ref{t:main_intr'}. 
$\hfill\Box$


\section{General problem}
\label{sec:general_problem}

Given a set $A\subseteq \F$ one can consider a general problem about 
finding good estimates for  
rational expressions $R(A)$ in terms of the sumsets and the product set of the set $A$. 
Namely, putting $K=|A+A|/|A|$ and $M=|AA|/A$ we can ask to 
seek 
a bound for cardinality of $R(A)$ of  the form $|R(A)| \ll_{K,M} |A|$. 
First such results were obtained in \cite{BKT}. 
For $R(A) = nA-mA$ or $R(A) = A^n/A^m$, where $n,m$ are positive integers such estimate exists and the corresponding  statement is called the Pl\"unnecke--Ruzsa inequality
as we have discussed in Section \ref{sec:definitions}. 
Moreover, thanks to the sum--product phenomenon \cite{BKT}, \cite{TV}, we know that in many fields $\F$ the following holds  $KM\gg |A|^c$, $c>0$ and hence a bound $|R(A)| \ll_{K,M} |A|$ trivially takes place (for large powers of $K$ and $M$).
Thus we need to specify here the dependence on $K$ and $M$. 
We can  suppose  that $R(A)$ simultaneously includes addition (subtraction) and multiplication (division) and hence it is naturally  to assume that the power of $K$ and $M$ in the presumable bound is at least one. Thus we have arrived to the following problem which we formulate for definiteness in the case of the simplest polynomial $R(x,y,z) = x(y+z)$.

\bigskip

{\bf Problem.} Suppose that  $A$ is a finite subset of 
$\mathbb{R}$ or suppose that $A$  is a sufficiently small set 
belonging to 
$\F_p$. 
Let $K=|A+A|/|A|$ and $M=|AA|/|A|$. Is it true that 
\begin{equation}\label{f:problem_K}
	|A(A+A)| \ll_M K |A| ?
\end{equation}

As we have seen in Section \ref{sec:small_AA} the answer to the dual question, namely, is it true that  $|AA+A| \ll_K M|A|$ is negative.  
Further if inequality \eqref{f:problem_K} takes place, then by the Cauchy--Schwarz inequality one has 
\begin{equation}\label{f:problem_K_r}
	\sum_{x} r^2_{A(A+A)} (x) \gg_M \frac{|A|^5}{K} \,.
\end{equation}
It is easy to see that a stronger form of bound \eqref{f:problem_K_r} follows from \eqref{f:popular_PA}, \eqref{f:popular_PA_1}.

\begin{proposition}
	Let $A\subseteq \F$ and $\eps \in \{-1,1\}$. 
	Then
\begin{equation}\label{f:r_1}
\sum_{x} r^2_{A^{\eps} (A\pm A)} (x) \gtrsim \frac{|A|^8 }{|AA^\eps|^2 |A\pm A|} \,,
\end{equation}
	and 
\begin{equation}\label{f:r_2}
	\sum_{x} r^2_{A^{\eps} (A\pm A)} (x) \gtrsim_{|AA^\eps|/|A|} (\E^{+}_{3/2} (A))^2 \,.
\end{equation}
\end{proposition}
\begin{proof} 
	Take $\D_* = |A|^2/(2|A\pm A|)$. 
	Using the pigeonhole principle, find $\D \ge \D_*$ and a set $P = \{ x ~:~ \D < r_{A\pm A} (x) \le 2\D\}$ such that $\sum_{x\in P} r_{A\pm A} (x) \gtrsim |A|^2$.
	Applying \eqref{f:popular_PA_1},
	we obtain $|PA^{\eps}| \lesssim M^2 |P|$, where $M=|AA^\eps|/|A|$. 
	Using the Cauchy--Schwarz inequality, we get 
$$ 
	\E^\times (P,A^{\eps}) \ge \frac{|A|^2 |P|^2}{|PA^{\eps}|} \gtrsim  \frac{|A|^2 |P|}{M^2} 
$$
	and multiplying the last estimate by $\D^2$, we arrive to
\begin{equation}\label{tmp:06.11_2}
	\sum_{x} r^2_{A^{\eps}(A+A)} (x) \ge \D^2 \E^\times (P,A^{\eps}) \ge  \frac{|A|^2 \D^2 |P|}{M^2} \gtrsim   \frac{|A|^4 \D_*}{M^2} \ge \frac{|A|^6 }{2M^2 |A\pm A|}   
\end{equation}
as required.
	To obtain \eqref{f:r_2} just use estimate \eqref{tmp:06.11_2} and inequality \eqref{f:popular_PA} with $k=3/2$. 
This completes the proof. 
$\hfill\Box$
\end{proof}


\begin{remark}
	With some efforts one can clean the logarithms in \eqref{f:r_1}, using the same scheme of the proof and more accurate but rather lengthy combinatorial computations. 
	We leave it for the interested reader, preferring to have a short proof with slightly worse estimates.  
\end{remark}


Estimates \eqref{f:r_1}, \eqref{f:r_2} are sharp as one can see taking $A$ with small product set $AA$.
Now we obtain another lower bound for $\sum_{x} r^2_{A(A+A)} (x)$ which is sharp, in contrary, for sets with small sumset $A+A$. 

\begin{proposition}
	Let $A, B\subseteq \F$ be finite sets. 
	Then 
\[
	|A/B| \cdot \sum_{s} r^2_{(A\pm B)/B} (s) \ge \E^{+} (A,B)^2 \,.
\]
\end{proposition}
\begin{proof}
	Take $s\in A \pm  B$.
	Then there are $n(s)$ pairs $(a_i,b_i) \in A \times B$  such that $s=a_i \pm b_i$,  $i \in \{1,\dots, n(s) \}$.
	Clearly, $\sum_s n^2 (s) = \E^{+} (A,B)$. 
	Consider the map  $\_phi : A\pm B \to 2^{A\times B \times B}$ defined as $\_phi(s) = \{ (a_i,b_i,b_j) ~:~ i, j\in \{1,\dots, n(s) \} \}$ and thus for any $s$, we have $|\_phi (s)| = n^2 (s)$. 
	Obviously, $\_phi(s)  = \_phi(s')$ implies that  $s=s'$ and $i=i'$, $j=j'$.
	Hence there are  $\E^{+} (A,B)$ such triples $(a_i,b_i,b_j)$.
	Further 
\[
	\frac{a_i \pm b_i}{b_j} = \frac{s}{b_j} = \frac{a_j \pm b_j}{b_j} = \frac{a_j}{b_j} \pm 1 \in \frac{A}{B} \pm 1  
\]
	and thus the image of the function $f(x,y,z) = (x \pm y)/z$ on our triples has cardinality at most $|A/B|$. 
	By the Cauchy--Schwarz inequality, we see that
	\[
		\sum_{s} r^2_{(A\pm B)/B} (s) \ge \sum_{\a \in \F} |\{ x\in A,\, y,z \in B ~:~ f(x,y,z) = \a \}|^2 \ge \E^{+} (A,B)^2 / |A/B|  
	\]
	as required. 
$\hfill\Box$
\end{proof}

\bigskip

Let us make a final observation. 
As we have seen at the beginning of  Section \ref{sec:small_AA} if $P = \{ s\in A\pm A ~:~ r_{A\pm A} (s) \ge \D \}$, then $\D |AP| \le M^2 |A|^2$. 
In other words, popular sets (in terms of $r_{A\pm A} (s)$) have small product or ratio with $A$. 
Interestingly,  that if we put now 
$$\t{P} = \{ s\in A\pm A ~:~ \exists  x, y \in A,\, x \pm  y = s,\, r_{A/A} (x / y) \ge \D \} \,,$$ 
i.e. $\t{P}$ is popular in terms of ratios, then a similar bound takes  place.
Indeed, put $\t{\Lambda} = \{\lambda \in A/A ~:~ r_{A/A} (\lambda) \ge \D \}$. 
We have $a(b\pm c) = ab(1\pm \frac{c}{b})$ and hence $|A\t{P}| \le |AA(1 \pm \t{\Lambda})|$. 
But, clearly, the map $\_phi : AA(1  \pm  \t{\Lambda}) \to AA/A \times (A\pm A)$ defined as $\_phi (x) = (\pi(x)/c(x), b(x) \pm c(x))$, where for $x\in AA(1 \pm  \t{\Lambda})$  we have put $\pi(x) \in AA$ and 
$b(x)/c(x) = \lambda(x) \in \t{\Lambda}$ 
is injective (consider the product of its coordinates). 
Thus by the Pl\"unnecke--Ruzsa inequality, we get
\[
	\D |A\t{P}| \le \D |AA(1+\t{\Lambda})| \le |AA/A| |A \pm A| \le |AA|^3 |A \pm A| / |A|^2 \le M^3 |A| |A \pm A|  
\]
as required.

\bigskip

\noindent{Olmezov K.I.\\
	Steklov Mathematical Institute,\\
	ul. Gubkina, 8, Moscow, Russia, 119991}\\
and 
\\
MIPT, \\ 
Institutskii per. 9, Dolgoprudnii, Russia, 141701\\
{\tt olmezov.ki@gmail.com}

\bigskip

\noindent{Semchankau A.S\\
	Steklov Mathematical Institute,\\
	ul. Gubkina, 8, Moscow, Russia, 119991}\\
{\tt aliaksei.semchankau@gmail.com}

\bigskip

\noindent{I.D.~Shkredov\\
Steklov Mathematical Institute,\\
ul. Gubkina, 8, Moscow, Russia, 119991}
\\
and
\\
IITP RAS,  \\
Bolshoy Karetny per. 19, Moscow, Russia, 127994\\
and 
\\
MIPT, \\ 
Institutskii per. 9, Dolgoprudnii, Russia, 141701\\
{\tt ilya.shkredov@gmail.com}

\end{document}